\documentclass[reqno]{amsart}
\usepackage{amsfonts}
\usepackage{graphicx}
\usepackage{amscd}
\usepackage{amsmath}
\usepackage{amssymb}
\usepackage[latin1]{inputenc}
\usepackage{amssymb,amsmath}
\usepackage{amsfonts}
\usepackage{graphicx,color}
\usepackage{amsthm}
\usepackage{graphicx}
\usepackage{fancyhdr}
\usepackage{hyperref}
\usepackage{color}
\theoremstyle{plain}

\newtheorem{lemma}{\bf Lemma}

\newtheorem{proposition}{\bf Proposition}
\newtheorem{remark}{Remark}

\newtheorem{theorem}{\bf Theorem}
\newtheorem{theorem*}{Theorem}{}
\numberwithin{equation}{section}
\newcommand\dv{\mathrm{div}}
\newcommand\dm{\mathrm{dm}}
\newcommand\dn{\mathrm{d\mu}}
\begin{document}
\title[Hadamard Type Variation Formulas]{Hadamard Type Variation Formulas for the Eigenvalues of the $\eta$-Laplacian and Applications}
\author{J.N.V. Gomes$^1$, M.A.M. Marrocos$^{2}$}
\author{R.R. Mesquita$^3$}

\address{$^{1,3}$Departamento de Matem\'atica, Instituto de Ci\^encias Exatas, Universidade Federal do Amazonas, Av. General Rodrigo Oct\'avio, 6200, 69080-900 Manaus, Amazonas, Brazil.}
\address{$^2$CMCC-Universidade Federal do ABC, Av. dos Estados, 5001, 09210-580 Santo Andr\'e, S\~ao Paulo, Brazil.}

\email{$^1$jnvgomes@pq.cnpq.br,jnvgomes@gmail.com}
\email{$^2$marcus.marrocos@ufabc.edu.br}
\email{$^3$mesquitaraul532@gmail.com}

\urladdr{$^{1,3}$http://www.mat.ufam.edu.br}
\urladdr{$^2$http://cmcc.ufabc.edu.br}

\thanks{$^1$Partially supported by CNPq-Conselho Nacional de Desenvolvimento Cient\'ifico e Tecnol\'ogico, of
the Ministry of Science, Technology and Innovation of Brazil}
\thanks{$^2$Partially supported by grant 2016/10009-3 São Paulo Research Foundation (FAPESP)}



\begin{abstract}
We consider an analytic family of Riemannian metrics on a compact smooth manifold $M$. We assume the Dirichlet boundary condition for the $\eta$-Laplacian, and obtain Hadamard type variation formulas for analytic curves of eigenfunctions and eigenvalues. As an application, we show that for a subset of all $C^r$~Riemannian metrics $\mathcal{M}^r$ on $M$, all eigenvalues of the $\eta$-Laplacian are generically simple, for $2\leq r< \infty$. This implies the existence of a residual set of metrics in $\mathcal{M}^r$, which makes the spectrum of the $\eta$-Laplacian simple. Likewise, we show that there exists a residual set of drifting functions $\eta$ in the space $\mathcal{F}^r$ of all $C^r$~functions on $M$, which makes again the spectrum of the $\eta$-Laplacian simple, for $2\leq r< \infty$. Besides, we give a precise information about the complementary of these residual sets, as well as about the structure of the set of deformations of a Riemannian metric (respectively of the set of deformations of a drifting function) which preserves double eigenvalues. Moreover, we consider a family of perturbations of a domain in a Riemannian manifold, and obtain Hadamard type formulas for the eigenvalues of the $\eta$-Laplacian in this case. We also establish generic properties of eigenvalues in this context.
\end{abstract}
\maketitle

\section{Introduction}
In \cite{berger}~Berger derived variation formulas for the eigenvalues of the Laplace-Beltrami operator with respect to differentiable one-parameter family of Riemannian metrics $g(t)$ on a smooth manifold $M$. Such formulas are known as Hadamard type variation formulas. In a seminal work, Uhlenbeck~\cite{Uhlenbeck} proved important results on generic properties of the eigenvalues and eigenfunctions of the Laplace-Beltrami operator $\Delta_{g}$ on a compact Riemannian manifold $(M,g)$ without boundary. In order to prove her results on genericity of the eigenvalues of $\Delta_g$ she used the Thom transversality theorem.

Here we work in the Dirichlet problem for the $\eta$-Laplacian $L_g:=\Delta_g-g(\nabla\eta,\nabla\cdot)$ on a compact Riemannian manifold $(M,g)$, and our main tools are Hadamard type variation formulas, where the differentiable function $\eta:M\to\Bbb{R}$ is known as drifting function. Such formulas are the optimal device to apply Teytel's approach~\cite{teytel}. The crucial step in the work of Teytel has been to impose a condition which is closely related to the strong Arnold hypothesis~\cite{Arnold} for double eigenvalues, but significantly easier to check. More precisely, let $\mathcal{M}^r$ denote the space of all $C^{r}$~Riemannian metrics on $M$ equipped with the $C^r$ topology, for $2\leq r <\infty$, and let $\Gamma$ be the set of all $g\in\mathcal{M}^r$ such that the eigenvalues of $L_g$ are all simple, so that each $g\in\Gamma$ can be obtained as a generic member of a differentiable family of self-adjoint operators $A(q)$ indexed by a parameter $q\in\mathcal{M}^r$. In this setting, we know a precise information about the complementary of $\Gamma$ as well as about the structure of the set of deformations of a Riemannian metric $g$ which preserves double eigenvalues of $L_g$, see Gomes and Marrocos~\cite{Gomes-Marrocos} for details, or the Teytel's paper for an abstract setting.

The use of Hadamard type formulas is common in the works of some authors. We would like to mention the works of Albert~\cite{A}, El Soufi and Ilias~\cite{ahmad}, Henry~\cite{henry} and Pereira~\cite{pereira}, which are good references in the literature on this topic. More recently, these formulas has been used by two first authors in \cite{Gomes-Marrocos2} to show a density theorem for a class of warping functions which makes the spectrum of the Laplacian a warped-simple spectrum. As an application, they gave a partial answer to a question posed by Zelditch~\cite{zelditch} about the generic situation of multiplicity of the eigenvalues of the Laplacian on principal bundles.

In this paper, we derive generic properties of the eigenvalues of the $\eta$-Laplacian with respect to variation metrics and/or drifting functions. We also work with perturbations of a bounded domain $\Omega$ (given by diffeomorphisms) in a Riemannian manifold $(M,g)$ and establish generic properties of eigenvalues with respect to these perturbations. For this, we consider a family of operators $\eta(t)$-Laplacian where the drifting function $\eta$ depends on the parameter $t$, see equations~\eqref{Lbarra} and \eqref{Eq-ProBar-L}. Besides, we consider a family $\{L_\eta\}$ of $\eta$-Laplacians parameterized by drifting functions $\eta$ in order to obtain analogous results as in~\cite[Sections~5 and 6]{Gomes-Marrocos}, see Theorem~\ref{thm-f}.

Before stating our theorems, we recall that a subset $\Gamma\subset\mathcal{M}^{r}$ is called residual if it contains a countable intersection of open dense sets. The property of metrics in $\Gamma$ is called generic if it holds on a residual subset.

In the following, we assume all manifolds to be oriented and those that are compact are assumed to have a boundary.

\begin{theorem}\label{thm-A}
Given a compact smooth manifold $M^n$, $n\geq2$, there exists a residual subset $\Gamma\subset\mathcal{M}^{r}$, $2\leq r< \infty$, such that for all $g\in \Gamma$ the eigenvalues of the Dirichlet problem for the $\eta$-Laplacian $L_g$ are simple.
\end{theorem}

Let $\Omega$ be a bounded domain in a Riemannian manifold $(M, g)$, and $D^r(\Omega)$ (with the fixed $C^r$~topology, $2\leq r<\infty$) be the set of all $f:\Omega\to M$ which are $C^r$~diffeomorphisms from $\Omega$ to $f(\Omega)$. It is known that this set is an affine manifold of a Banach space (see~\cite{delfour}). Then, we show that the following property is generic:

\begin{theorem}\label{thm-B}
Given a Riemannian manifold $(M^n,g)$, $n\geq2$, and define $\mathfrak{D}\subset D^r(\Omega)$, $2\leq r< \infty$, to be the subset of $f:\Omega\to (M,g)$ such that all eigenvalues of the $\eta$-Laplacian $L_g$ on $C_c^{\infty}(f(\Omega))$ (with Dirichlet boundary condition on $f(\Omega)$) are simple. Then $\mathfrak{D}$ is residual.
\end{theorem}

Let $\mathcal{F}^r$ (with the fixed $C^r$~topology, $2\leq r<\infty$) be the set of all $C^r$ drifting functions $\eta$, and let us use the notation $L_\eta:=\Delta_g-g(\nabla\eta,\nabla\cdot)$ to emphasize that the parameter is $\eta.$

\begin{theorem}\label{thm-f-A}
Given a compact Riemannian manifold $(M^n,g)$, $n\geq2$, there exists a residual subset $\mathcal{E}\subset\mathcal{F}^{r}$, $2\leq r< \infty$, such that for all $\eta\in\mathcal{E}$, the eigenvalues of the Dirichlet problem for $L_\eta$ are simple.
\end{theorem}

Now, we discuss an interesting case of the spectrum of $L_\eta$ which stems from work of Teytel~\cite{teytel}.

\begin{theorem} \label{thm-f}
Let $(M^n,g)$, $n\geq2$, be a compact Riemannian manifold, and let $\mathcal{E}\subset\mathcal{F}^{r}$, $2\leq r< \infty$, be a residual subset such that for all $\eta\in\mathcal{E}$, the eigenvalues of the Dirichlet problem for $L_\eta$ are simple.
\begin{enumerate}
\item The set $\mathcal{F}^r\backslash\mathcal{E}$ has meager codimension $2$ in $\mathcal{F}^r.$
\item Take $\eta_0\in\mathcal{F}^r$, and let $\lambda$ be an eigenvalue of the operator $L_{\eta_0}$ of multiplicity $2$. Then, in a neighborhood of $\eta_0$, the set of all $\eta\in\mathcal{F}^r$ such that $L_\eta$ admits an eigenvalue $\lambda(\eta)$ of multiplicity $2$ near $\lambda$ form a submanifold of meager codimension $2$ in $\mathcal{F}^{r}$.
\item Consider the same set up as Item $(2)$. Then, in a neighborhood of $\eta_0$, the set of all $\eta\in\mathcal{F}^r$ which preserves double eigenvalues, i.e., $L_\eta$ admits an eigenvalue $\lambda(\eta)$ of multiplicity $2$ such that $\lambda(\eta)=\lambda(\eta_0)$, form a submanifold of meager codimension $3$ in $\mathcal{F}^{r}$.
\end{enumerate}
\end{theorem}

\begin{remark}
It is important to observe that Theorem~\ref{thm-f} is also true in the context of the family of bounded domains in a Riemannian manifold. The proof follows the same steps as in the proof of Theorem~\ref{thm-f}.
\end{remark}

\section{Preliminaries}
Let us consider an oriented compact Riemannian manifold $(M,g)$ with boundary $\partial M$ and volume form $dV$. It is endowed with a weighted measure of the form $\dm=e^{-\eta}dV$, where $\eta:M\to\mathbb{R}$ is a differentiable function. Let $\mathrm{d\nu}$ be the volume form induced on $\partial M$ and $\dn=e^{-\eta}\mathrm{d\nu}$ be the corresponding weighted measure on $\partial M$. We define the $\eta$-Laplacian by $L_{g}=\Delta_{g}-g(\nabla\eta,\nabla \cdot)$ which is essentially self-adjoint on $C_c^{\infty}(M)$. Observe that this allows us to use perturbation theory for linear operators \cite{Kato}. To do this, we consider the set $\mathcal{M}^r$ of all $\mathcal{C}^r$ Riemannian metrics on $M$. Then every $g\in\mathcal{M}^r$ determines the sequence $0=\mu_{0}(g)<\mu_{1}(g)\leq\mu_{2}(g)\leq\cdots\leq\mu_{k}(g)\leq\cdots$ of the eigenvalues of $L_{g}$ counted with their multiplicities. We regard each eigenvalue $\mu_{k}(g)$ as a function of $g$ in $\mathcal{M}^r$. Note that, in general, the functions $g\mapsto\mu_k(g)$ are continuous but not differentiable (see~\cite{Kato}). They are differentiable when $\mu_k$ is simple. With these notations, the divergence theorem remains valid under the form $\int_{M}L_gf\dm=\int_{\partial M}g(\nabla f,\nu)\dn$. Thus, the integration by parts formula is given by
\begin{equation}\label{int_pp}
\int_{M}\ell L_gf\dm=-\int_{M}g(\nabla\ell,\nabla f)\dm+\int_{\partial M}\ell g(\nabla f,\nu)\dn
\end{equation}
for all $f,\ell\in C^{\infty}(M).$

The inner product induced by $g$ on the space of $(0,2)$--tensors on $M$ is given by $\langle T,S\rangle=\mathrm{tr}\big(TS^{*}\big)$, where $S^*$ denotes the adjoint tensor of $S$. Clearly, we get in local coordinates
\begin{equation*}
\langle T,S\rangle = \sum_{i,j,k,l}g^{ik}g^{jl}T_{ij}S_{kl}.
\end{equation*}
Furthermore, we have $\Delta_gf=\langle\nabla^2f,g\rangle$ where $\nabla^2f=\nabla df$ is the Hessian of $f$. We also recall that each $(0,2)$--tensor $T$ on $(M,g)$ can be associated to a unique $(1,1)$--tensor by
$g(T(Z),Y) := T(Z,Y)$ for all $Y,Z\in\mathfrak{X}(M)$. We shall slightly abuse notation here and will also write $T$ for this $(1,1)$--tensor. So, we consider the $(0,1)$--tensor given by
\begin{equation*}
(\dv T)(v)(p) = \mathrm{tr}\big(w \mapsto (\nabla_w T)(v)(p)\big),
\end{equation*}
where $p\in M$ and $v\in T_pM.$ Moreover, we can define a $(0,1)$--tensor $\dv_{\eta}T$ putting $\dv_{\eta}T:=\dv T-d\eta\circ T.$

Before proving our main results, we present the following one.

\begin{lemma}\label{lem1}
Let $T$ be a symmetric $(0,2)$--tensor on a Riemannian manifold $(M,g)$. Then
\begin{equation}
\label{div}\dv_{\eta}(T(\varphi Z))= \varphi\langle \dv_{\eta}T,Z\rangle+ \varphi\langle \nabla Z, T\rangle + T(\nabla\varphi ,Z),
\end{equation}
for each $Z\in\mathfrak{X}(M)$ and any differentiable function $\varphi$ on $M$.
\end{lemma}
\begin{proof} Let $\{e_1,\ldots,e_n\}$ be a local orthonormal frame on $(M,g)$. Using the properties of $\dv_{\eta}$ and the symmetry of $T$, for each $Z\in\mathfrak{X}(M)$ and any differentiable function $\varphi$ on $M,$ we compute
\begin{eqnarray*}
\dv_{\eta} (T(\varphi Z))&=& \varphi \dv_{\eta}(T(Z))+ g(\,\nabla\varphi,T(Z)\,)\\
&=&\varphi(\dv T)(Z)+ \varphi\sum_{i} g(T(\nabla_{e_i}Z) , e_i) -\varphi g(\nabla\eta,T(Z)) + T(\nabla\varphi,Z)\\
&=&\varphi(\dv_{\eta} T)(Z)+ \varphi\langle \nabla Z, T\rangle + T(\nabla\varphi,Z).
\end{eqnarray*}
To complete the proof is sufficient to use the duality $(\dv_{\eta} T)(Z)=\langle \dv_{\eta} T, Z\rangle$.
\end{proof}

Let us observe that for every $X\in\mathfrak{X}(M)$ the operator $\dv_{\eta}X=\dv X-g(\nabla\eta,X)$ has the same properties of the operator $\dv X$ as well as is valid
$\int_M\dv_{\eta}X\dm=\int_{\partial M}g(X,\nu)\dn.$

\section{Hadamard Type Variation Formulas}
In this section, we consider a differentiable variation $g(t)$  of the metric $g$, so that $(M,g(t), dm_t)$ is a Riemannian manifold with a differentiable measure. Denoting by $H$ the $(0,2)$--tensor given by $H_{ij}=\frac{d}{dt}\big|_{t=0}g_{ij}(t)$ and writing $h=\langle H,g\rangle$, we easily get $\frac{d}{dt}\big|_{t=0}\dm_t=\frac{1}{2}h\dm$. From now on we shall write in local coordinates $f_i=\partial_i f$. We first prove the following lemma.

\begin{lemma}\label{lem2}
Let $(M,g)$ be a Riemannian manifold and $g(t)$ be a differentiable variation of the metric $g$. Then, for all $f\in C_c^{\infty}(M)$, we have
\begin{equation}
L'f=\langle\frac{1}{2}dh-\dv_{\eta}H,df\rangle-\langle H,\nabla^2f\rangle,
\end{equation}
where $L':=\frac{d}{dt}\big|_{t=0}L_{g(t)}.$
\end{lemma}
\begin{proof}
Indeed, since $\langle df,d\ell\rangle=g^{ij}(t)f_i\ell_j,$ for any $\ell\in C_c^{\infty}(M),$ and $\frac{d}{dt}\big|_{t=0}g^{ij}(t)=-g^{ik}g^{js}H_{ks}$, we have
\begin{equation}\label{eq1lem1}
\frac{d}{dt}\Big|_{t=0}\langle df,d\ell\rangle=-g^{ik}g^{js}H_{ks}f_i\ell_j=-H(g^{ik}f_i\partial_k , g^{js}\ell_j\partial_s)=-H(\nabla f,\nabla \ell).
\end{equation}
From integration by parts formula, we get
\begin{equation*}
\int_{M}\ell L_{g(t)}f\dm_t=-\int_{M}\langle d\ell,df\rangle \dm_t.
\end{equation*}
Hence, from equation \eqref{eq1lem1}, we have at $t=0$
\begin{equation}\label{eq2lem1}
\int_M\ell L'f\dm + \frac{1}{2}\int_M\ell hLf\dm=\int_MH(\nabla f,\nabla \ell)\dm - \frac{1}{2}\int_Mh\langle d\ell,df\rangle \dm.
\end{equation}
Applying Lemma~\ref{lem1} for $T=H$, $\varphi=\ell$ and $Z=\nabla f$ we have
\begin{equation}\label{eq3lem1}
\dv_{\eta}(H(\ell\nabla f))=\ell\langle \dv_{\eta}H,df\rangle+\ell\langle H,\nabla^2f\rangle+H(\nabla f,\nabla \ell).
\end{equation}
Moreover,
\begin{equation}\label{eq4lem1}
\dv_{\eta}(\ell h\nabla f)=\ell hLf+\ell\langle dh,df\rangle+h\langle d\ell,df\rangle.
\end{equation}
Hence, plugging \eqref{eq3lem1} and \eqref{eq4lem1} into \eqref{eq2lem1}, we obtain
\begin{equation}
\int_{M}\ell L'f \dm=\int_M \ell\Big(\frac{1}{2}\langle dh,df\rangle- \langle \dv_{\eta}H,df\rangle-\langle H,\nabla^2f\rangle\Big)\dm,
\end{equation}
which concludes the proof of the lemma.
\end{proof}

Next, we consider a differentiable function $\eta:I\times M\to\Bbb{R}$ and  write for simplicity  $\dot{\eta}=\frac{d}{dt}\big|_{t=0}\eta(t)$. For all $f\in C^{\infty}(M)$ we define
\begin{equation}\label{Lbarra}
\bar{L}_{t}f:=\Delta_{t}f - g(t)(\nabla\eta(t),\nabla f).
\end{equation}
Thus,
\begin{eqnarray}
\nonumber\frac{d}{dt}\Big|_{t=0}\bar{L}_{t}f &=& \Delta' f - \frac{d}{dt}\Big|_{t=0}g^{ij}(t)\eta_{i}(t)f_{j}\\
\nonumber&=&\Delta' f -\Big(\frac{d}{dt}\Big|_{t=0}g^{ij}(t)\Big)\eta_{i}f_{j}- g^{ij}\partial_i\frac{d}{dt}\Big|_{t=0}\eta(t) f_j\\
\label{AAS}&=& L'f - \langle\nabla\dot{\eta},\nabla f\rangle.
\end{eqnarray}

The next result extends Lemma~$3.15$ of Berger~\cite{berger} for the $\eta$-Laplacian. Firstly, we note that given an eigenvalue $\lambda(g_0)$ of $\bar{L}_{g_0}$ with multiplicity $m(\lambda(g_0))$, there are a positive number $\epsilon_{\lambda(g_0), g_0}$ and a neighborhood $\mathcal{V}_{\epsilon}$ in $\mathcal{M}^r$, $2\leq r<\infty$, such that for all $g\in \mathcal{V}_{\epsilon}$ one has
\begin{equation}\label{kato-continuity}
\sum_{\{|\lambda-\lambda(g_0)|<\epsilon_{\lambda(g_0), g_0}\}\cap spec(\bar{L}_g)}m(\lambda)=m(\lambda(g_0)).
\end{equation}
Indeed, Eq.~\eqref{kato-continuity} is a consequence of the continuity of a finite system of eigenvalues, see \cite[Section 5, Chapter 4]{Kato}. In this setting, we prove the following generic result.
\begin{proposition}\label{PropExist}
Let $(M,g)$ be a compact Riemannian manifold. Consider a real analytic one-parameter family of Riemannian metrics $g(t)$  on $M$ with $g=g(0).$ If $\lambda$ is an eigenvalue of multiplicity $m>1$ for the $\eta$-Laplacian $L_{g}$, then there exists $\varepsilon>0$, and there exist scalars $\lambda_{i}(t)$ ($i=1,\dots,m$) and functions $\phi_{i}(t)$  varying analytically in $t$, such that, for all $|t|<\varepsilon$ the following relations hold:
\begin{enumerate}
\item $L_{g(t)}\phi_{i}(t)=\lambda_{i}(t)\phi_{i}(t)$;
\item $\lambda_{i}(0)=\lambda$;
\item $\{\phi_{i}(t)\}$ is orthonormal in $L^2(M,\dm_t)$.
\end{enumerate}
\end{proposition}
\begin{proof}
First, let us consider an extension $g(z)$ of $g(t)$ to a domain $D_0$ of the complex plane $\mathbb{C}$. So, we consider the operator $$L_{g(z)}:\mathcal{C}^{\infty}(M;\mathbb{C})\rightarrow \mathcal{C}^{\infty}(M;\mathbb{C}),$$ that in a local coordinate system is given by
\begin{equation}\label{EQ24}
L_{g(z)}f=g^{ij}(z)\left(\frac{\partial^2f}{\partial x_i\partial x_j}-\Gamma_{ij}^k(z)\frac{\partial f}{\partial x_k}-\frac{\partial \eta}{\partial x_i}\frac{\partial f}{\partial x_j}\right)
\end{equation}
for all $f\in \mathcal{C}^{\infty}(M;\mathbb{C})$, with
$$\Gamma_{ij}^k=\frac{1}{2}g^{k\ell}\left(\frac{\partial g_{i\ell}}{\partial x_j}+\frac{\partial g_{j\ell}}{\partial x_i}-\frac{\partial g_{ij}}{\partial x_{\ell}}\right).$$
Now we observe that the domain $D=H^2(M)\cap H_0^1(M)$ of the operator $L_{g(z)}$ is independent of $z$, since $M$ is compact, any two metrics are equivalent. Besides, the application $z\mapsto L_{g(z)}f$ is holomorphic for $z\in D_0$ and for every $f\in D$. Thus, $L_{g(z)}$ is a holomorphic family of type $(A)$ in~\cite{Kato}. Now we need to prove that the operator $L_{g(z)}$ is self-adjoint with fixed inner product, for this purpose, for each $t$, we can construct an isometry
\begin{equation*}
P:L^2(M,\dm)\to L^2(M,\dm_t)
\end{equation*}
taking, for each $u$, $P(u)=\frac{\sqrt[4]{det(g_{ij}(0))}}{\sqrt[4]{det(g_{ij}(t))}}u$. In fact,
\begin{equation*}
\int_{M}P(u)P(v)\dm_{t}=\int_{M}\frac{\sqrt{det(g_{ij}(0))}}{\sqrt{det(g_{ij}(t))}}uv\dm_{t}=\int_{M}uv\dm.
\end{equation*}
Thus, the operator $\tilde{L_{t}}:=P^{-1}\circ L_{t}\circ P$ will have the same eigenvalues of $L_t:H^2(M,\dm_t)\cap H_{0}^{1}(M,\dm_t)\to L^2(M,\dm_t)$.
But the compactness of the $M$ implies that $\tilde{L_{t}}$ is self-adjoint, since
\begin{eqnarray*}
\int_{M}v\tilde{L_{t}}u\dm&\stackrel{(isom.)}{=}&\int_{M}P(v)L_{t}P(u)\dm_{t}=\int_{M}P(u)L_{t}P(v)\dm_{t}\\
&\stackrel{(isom.)}{=}&\int_{M}P^{-1}P(u)P^{-1}L_{t}P(v)\dm=\int_{M}u\tilde{L_{t}}v\dm.
\end{eqnarray*}
Under these conditions we can apply a theorem due to Rellich~\cite{Rellich} or Theorem~3.9 in Kato~\cite{Kato}, to obtain the result of the proposition.
\end{proof}

We observe that the same result of Proposition~\ref{PropExist} holds for $\bar{L}_t$ operator defined in \eqref{Lbarra}. Now, we will derive the first Hadamard type variation formula which generalizes substantially one of Berger's formulas~\cite{berger}.

\begin{proposition}\label{pro bar-L}
Let $(M,g)$ be a compact Riemannian manifold and $g(t)$ be a differentiable variation of the metric $g$, $\phi_i(t)\in C^{\infty}(M)$ a differentiable family of functions and $\lambda(t)$ a differentiable family of real numbers such that $\lambda_i(0)=\lambda$ for each $i=1,\ldots,m$ and for all $t$
$$\left\{
\begin{array}{cccc}
    -\bar{L}_t\phi_i(t) &=& \lambda_i(t)\phi_i(t) & \hbox{in} \,\,\;M\\
    \phi_i(t)&=&0 & \hbox{on}\,\,\; \partial M,
\end{array}
\right.$$
with $\langle\phi_i(t),\phi_j(t)\rangle_{L^2(M,\dm_t)}=\delta_{ij}$. Then the derivative of the $t \mapsto (\lambda_i(t)+\lambda_j(t))\delta_{ij}$ is given by
\begin{equation}\label{Eq-ProBar-L}
(\lambda_i+\lambda_j)'\delta_{ij} = \int_{M}\big\langle\frac{1}{2}L(\phi_i\phi_j)g - 2d\phi_i\otimes d\phi_j, H\big\rangle \dm + \int_{M}\langle\nabla\dot{\eta},\nabla(\phi_{i}\phi_{j})\rangle\dm.
\end{equation}
\end{proposition}

\begin{proof}
We begin by proving the case when $\eta$ does not depend on $t$. Differentiating the equation $-L_{g(t)}\phi_i(t)=\lambda_i(t)\phi_i(t)$, we have at $t=0$, $-L'\phi_i-L\phi'_i=\lambda_i'\phi_i+\lambda_i\phi'_i,$ so
\begin{equation*}
-\int_{M}(\phi_jL'\phi_i+\phi_jL\phi'_i)\dm =\int_{M}(\lambda_i'\phi_i\phi_j+\lambda_j\phi_j\phi_i') \dm = \lambda_i'\int_{M}\phi_j\phi_i\dm-\int_{M}\phi_{i}'L\phi_{j}\dm.
\end{equation*}
Using integration by parts formula and the fact that $\phi_i=0$ on $\partial M$, we obtain
\begin{equation*}
\lambda_i'\delta_{ij}= - \int_{M}\phi_jL'\phi_i \dm.
\end{equation*}
Thus, writing $s_{ij}=(\lambda_i'+\lambda_j')$ we deduce from Lemma~\ref{lem2}
\begin{eqnarray*}
-s_{ij}\delta_{ij} &=& \int_{M}\phi_jL'\phi_i \dm + \int_{M}\phi_iL'\phi_j \dm\\
&=&\int_M(\langle \frac{1}{2}dh-\dv_{\eta}H,\phi_jd\phi_i + \phi_id\phi_j \rangle - \langle H,\phi_j\nabla^2\phi_i + \phi_i\nabla^2\phi_j\rangle)\dm\\
&=&\int_{M}\langle \frac{1}{2}dh ,d(\phi_i\phi_j)\rangle \dm - \int_{M}\phi_j\big(\langle \dv_{\eta}H, d\phi_i\rangle+\langle H,\nabla^2\phi_i\rangle\big) \dm\\
&&-\int_{M}\phi_i\big(\langle \dv_{\eta}H,d\phi_j\rangle+\langle H,\nabla^2\phi_j\rangle\big) \dm.
\end{eqnarray*}
We next use Lemma \ref{lem1} and again integration by parts formula to get
\begin{eqnarray}\label{eqL}
-s_{ij}\delta_{ij} &=& - \int_{M}\frac{h}{2}L(\phi_i\phi_j)\dm + 2 \int_{M} H(\nabla\phi_i,\nabla\phi_j) \dm,
\end{eqnarray}
or equivalently
\begin{equation*}
s_{ij}\delta_{ij} = \int_{M}\big\langle\frac{1}{2}L(\phi_i\phi_j)g - 2d\phi_i\otimes d\phi_j, H\big\rangle \dm.
\end{equation*}
In the general case, we differentiate  the equation $-\bar{L}_t\phi_i(t)=\lambda_i(t)\phi_i(t)$ at  $t=0$, $-\bar{L}'\phi_i- L\phi'_i=\lambda_i'\phi_i+\lambda_i\phi'_i,$ so $-L'\phi_{i}-L\phi_{i}'=\lambda_i'\phi_{i}+\lambda_j\phi_{i}'-\langle\nabla\dot{\eta},\nabla \phi_i\rangle$. Thus, we have that
\begin{equation*}
\lambda_i'\delta_{ij}= - \int_{M}\phi_jL'\phi_i \dm + \int_{M}\phi_{j}\langle\nabla\dot{\eta},\nabla \phi_i\rangle\dm.
\end{equation*}
A calculation analogous to the one above completes the proof.
\end{proof}

Now, we show how to extend for the $\eta$-Laplacian a result by El Soufi and Ilias~\cite[Corollary~2.1]{ahmad}.

\begin{proposition}\label{pro difeo}
Let $(M,g)$ be a Riemannian manifold, $\Omega\subset M$ a bounded domain,  $f_{t}:\Omega\to (M,g)$ a analytic family of diffeomorphisms from $\Omega$ to $\Omega_{t}=f_{t}(\Omega)$, $f_0$ is the identity map and $\lambda$ an eigenvalue of multiplicity $m>1$. Then there exist an analytic family of $m$ functions $\phi_i(t)\in C^{\infty}(\Omega_{t})$ with $\langle\phi_i(t),\phi_j(t)\rangle_{L^2(\Omega_t,\dm)}=\delta_{ij}$ and real numbers $\lambda_i(t)$ with $\lambda_i(0)=\lambda$, such that, for all $t$ and $i=1,\ldots,m$, they are solutions for the Dirichlet problem
$$\left\{\begin{array}{ccccccc}
    -L\phi_i(t) &=& \lambda_i(t)\phi_i(t) & \;\;\Omega_{t}\\
    \phi_i(t) &=&0 & \;\; \partial\Omega_t.
\end{array}\right.$$
Moreover, the derivative of the curve $t \mapsto (\lambda_i(t)+\lambda_j(t))\delta_{ij}$ is given by
\begin{equation}\label{eq-prop3}
(\lambda_i+\lambda_j)'\delta_{ij}=-2\int_{\partial\Omega}\langle V,\nu\rangle\frac{\partial\phi_i}{\partial\nu}\frac{\partial\phi_j}{\partial\nu}\dn,
\end{equation}
where $\langle V, \nu\rangle=g(V, \nu)$ and $V=\frac{d}{dt}\big|_{t=0}f_{t}$.
\end{proposition}
\begin{proof}
We consider the family of metrics $g(t)=f_{t}^*g$ on $\Omega$. So, we can apply Proposition~\ref{PropExist} for $\bar{L}_t$. Considering
\begin{equation*}
\bar{L}_t(\phi_i(t)\circ f_t):=\Delta_t(\phi_i(t)\circ f_t)-g(t)(\nabla(\eta\circ f_t),\nabla(\phi_i(t)\circ f_t)),
\end{equation*}
we obtain
\begin{eqnarray*}
\bar{L}_t(\phi_i(t)\circ f_t)(p)=-\lambda_i(t)\phi_i(t)\circ f_t(p).
\end{eqnarray*}
For $\bar\phi_i(t)=\phi_i(t)\circ f_t$, we have $\forall t,$ $\langle\bar\phi_i(t),\bar\phi_j(t)\rangle_{L^2(\Omega,\dm_t)}=\delta_{ij}$ and
$$\left\{
\begin{array}{cccc}
-\bar{L}_t\bar\phi_i(t) &=& \lambda_i(t)\bar\phi_i(t) &\hbox{in}\,\, \Omega\\
\bar\phi_i(t) &=& 0 & \hbox{on}\,\, \partial\Omega.
\end{array}\right.$$
Since $\phi_{i}\circ f_0=\phi_i$ and $\eta(t)=\eta\circ f_t$, we have by Proposition~\ref{pro bar-L}
\begin{eqnarray*}
s_{ij}\delta_{ij} &=& \int_{\Omega}\frac{h}{2}L(\phi_i\phi_j)\dm - 2\int_{\Omega}H(\nabla\phi_i,\nabla\phi_j)\dm +\int_{\Omega}\langle\nabla\dot{\eta},\nabla(\phi_{i}\phi_{j})\rangle\dm.
\end{eqnarray*}
Recall that,  $H=\frac{d}{dt}\big|_{t=0}f_t^*g=\mathcal{L}_Vg$ where $V=\frac{d}{dt}\big|_{t=0}f_t$. Then
\begin{eqnarray*}
s_{ij}\delta_{ij}&=&\int_{\Omega}\frac{1}{2}L(\phi_i\phi_j)\langle g,H\rangle \dm-2\int_{\Omega}\Big(\frac{d}{dt}\Big|_{t=0}f_t^*g\Big)(\nabla\phi_i,\nabla\phi_j)\dm\\ &&+\int_{\Omega}\langle\nabla\dot{\eta},\nabla(\phi_{i}\phi_{j})\rangle\dm\\
&=&\int_{\Omega}L(\phi_i\phi_j)\dv V \dm-2\int_{\Omega}\langle\nabla_{\nabla\phi_i}V,\nabla\phi_j\rangle \dm\\
&& -2\int_{\Omega}\langle\nabla_{\nabla\phi_j}V,\nabla\phi_i\rangle \dm + \int_{\Omega}\langle\nabla\dot{\eta},\nabla(\phi_{i}\phi_{j})\rangle\dm.
\end{eqnarray*}
But,
\begin{eqnarray*}
\langle\nabla_{\nabla\phi_{i}}V, \nabla\phi_{j}\rangle = \dv_{\eta}(\langle V,\nabla\phi_{j}\rangle\nabla\phi_{i})+\lambda\langle V,\nabla\phi_{j}\rangle\phi_{i}-\nabla^{2}\phi_{j}(V,\nabla\phi_{i}).
\end{eqnarray*}
Since $\lambda=\lambda_i(0)=\lambda_j(0)$ and $\frac{s_{ij}}{2}\delta_{ij}=a_{ij}$ we have
\begin{eqnarray*}
a_{ij}&=& - \lambda\int_{\Omega}\phi_{i}\phi_{j}\dv V \dm+\int_{\Omega}\langle\nabla\phi_{i},\nabla\phi_{j}\rangle \dv V \dm-\int_{\Omega}\dv_{\eta}(\langle V,\nabla\phi_{j}\rangle\nabla\phi_{i})\dm\\
&&-\lambda\int_{\Omega}\langle V,\nabla\phi_{j}\rangle\phi_{i} \dm+\int_{\Omega}\nabla^{2}\phi_{j}(V,\nabla\phi_{i})\dm-\int_{\Omega}\dv_{\eta}(\langle V,\nabla\phi_{i}\rangle\nabla\phi_{j})\dm\\
&&- \lambda\int_{\Omega}\langle V,\nabla\phi_{i}\rangle \phi_{j}\dm+\int_{\Omega}\nabla^{2}\phi_{i}(V,\nabla\phi_{j})\dm+ \frac{1}{2}\int_{\Omega}\langle\nabla\dot{\eta},\nabla(\phi_{i}\phi_{j})\rangle\dm\\
&=&-\lambda\int_{\Omega}\Big(\phi_{i}\phi_{j}\dv V +\langle V,\nabla (\phi_{i}\phi_{j})\rangle\Big)\dm -\int_{\partial\Omega}\langle V,\nabla\phi_{j}\rangle\langle\nabla\phi_{i},\nu\rangle \dn\\
&&-\int_{\partial\Omega}\langle V,\nabla\phi_{i}\rangle\langle\nabla\phi_{j},\nu\rangle \dn + \int_{\Omega}\langle\nabla\phi_{i},\nabla\phi_{j}\rangle \dv V \dm +\int_{\Omega}\nabla^{2}\phi_{j}(V,\nabla\phi_{i})\dm\\
&&+\int_{\Omega}\nabla^{2}\phi_{i}(V,\nabla\phi_{j})\dm+ \frac{1}{2}\int_{\Omega}\langle\nabla\dot{\eta},\nabla(\phi_{i}\phi_{j})\rangle\dm.
\end{eqnarray*}
As $\phi_i=0 $ on $ \partial\Omega$, we have $\nabla\phi_i=\langle\nabla\phi_i,\nu\rangle\nu=\frac{\partial\phi_i}{\partial\nu}\nu$ on $\partial\Omega.$ Moreover,
\begin{eqnarray*}
\dv_{\eta}(\langle \nabla\phi_i,\nabla\phi_j \rangle V)+\langle \nabla\phi_i,\nabla\phi_j \rangle\langle\nabla\eta,V\rangle
&=&\dv(\langle \nabla\phi_i,\nabla\phi_j \rangle V)\\
&=&\langle\nabla\phi_{i},\nabla\phi_{j}\rangle \dv V+\langle\nabla\langle\nabla\phi_{i},\nabla\phi_{j}\rangle,V\rangle\\
&=&\langle\nabla\phi_{i},\nabla\phi_{j}\rangle \dv V+\nabla^{2}\phi_{j}(V,\nabla\phi_{i})\\
&&+\nabla^{2}\phi_{i}(V,\nabla\phi_{j}).
\end{eqnarray*}
So,
\begin{eqnarray*}
a_{ij}&=&-\lambda\int_{\Omega}\dv(\phi_{i}\phi_{j}V)\dm-2\int_{\partial\Omega}\langle V,\nu\rangle\frac{\partial\phi_{i}}{\partial\nu}\frac{\partial\phi_{j}}{\partial\nu}\dn+\int_{\Omega}\dv_{\eta}(\langle\nabla\phi_{i},\nabla\phi_{j}\rangle V)\dm\\
&&+\int_{\Omega}\langle\nabla\phi_{i},\nabla\phi_{j}\rangle\langle\nabla\eta,V\rangle \dm + \frac{1}{2}\int_{\Omega}\langle\nabla\dot{\eta},\nabla(\phi_{i}\phi_{j})\rangle\dm.
\end{eqnarray*}
It follows that
\begin{eqnarray*}
a_{ij}&=&-\lambda\int_{\Omega}\dv(\phi_{i}\phi_{j}V)\dm-2\int_{\partial\Omega}\langle V,\nu\rangle\frac{\partial\phi_{i}}{\partial\nu}\frac{\partial\phi_{j}}{\partial\nu}\dn+\int_{\partial\Omega}\langle V,\nu\rangle\frac{\partial\phi_{i}}{\partial\nu}\frac{\partial\phi_{j}}{\partial\nu}\dn\\
&&+\int_{\Omega}\langle\nabla\phi_{i},\nabla\phi_{j}\rangle\langle\nabla\eta,V\rangle \dm+ \frac{1}{2}\int_{\Omega}\langle\nabla\dot{\eta},\nabla(\phi_{i}\phi_{j})\rangle\dm\\
&=&-\int_{\partial\Omega}\langle V,\nu\rangle\frac{\partial\phi_{i}}{\partial\nu}\frac{\partial\phi_{j}}{\partial\nu}\dn-\lambda\int_{\Omega}\dv(\phi_{i}\phi_{j} V)\dm\\
&&+\int_{\Omega}\langle\nabla\phi_{i},\nabla\phi_{j}\rangle\langle\nabla\eta,V\rangle \dm+ \frac{1}{2}\int_{\Omega}\langle\nabla\dot{\eta},\nabla(\phi_{i}\phi_{j})\rangle\dm.
\end{eqnarray*}
On the other hand,
\begin{eqnarray*}
0=\int_{\Omega}\dv_{\eta}(\phi_{i}\phi_{j}V)\dm = \int_{\Omega}\dv(\phi_{i}\phi_{j}V)\dm-\int_{\Omega}\phi_{i}\phi_{j}\langle\nabla\eta,V\rangle \dm.
\end{eqnarray*}
Hence
\begin{eqnarray}\label{eq-a_ij}
 a_{ij}&=&-\int_{\partial\Omega}\langle V,\nu\rangle\frac{\partial\phi_i}{\partial\nu}\frac{\partial\phi_j}{\partial\nu}\dn\\
\nonumber&&+\int_{\Omega}\big(\langle\nabla\phi_{i},\nabla\phi_{j}\rangle-\lambda\phi_{i}\phi_{j}\big)\langle\nabla\eta,V\rangle\dm + \frac{1}{2}\int_{\Omega}\langle\nabla\dot{\eta},\nabla(\phi_{i}\phi_{j})\rangle\dm.
\end{eqnarray}
Since $\eta(t,p)=\eta\circ f(t,p)$ we have
\begin{equation*}\dot{\eta}=\frac{d}{dt}\Big|_{t=0}\eta(t,p)=\frac{d}{dt}\Big|_{t=0}(\eta\circ f)(t,p)=d\eta\Big|_{f(0,p)}\cdot\frac{d}{dt}\Big|_{t=0}f_t(p)=d\eta\Big|_p(V)=\langle\nabla\eta,V\rangle.
\end{equation*}
We next use that $\lambda_i(0)=\lambda_j(0)=\lambda$, $L(\phi_i\phi_j)=\phi_iL\phi_j +\phi_jL\phi_i+2\langle\nabla\phi_i,\nabla\phi_j\rangle$ and integration by parts formula to calculate
\begin{eqnarray*}
\frac{1}{2}\int_{\Omega}\langle\nabla\dot{\eta},\nabla(\phi_{i}\phi_{j})\rangle\dm &=&-\frac{1}{2}\int_{\Omega}\dot{\eta}L(\phi_i\phi_j)\dm+\frac{1}{2}\int_{\partial\Omega}\dot{\eta}\langle\nu,\nabla(\phi_i\phi_j)\rangle\dn\\
&=&\int_{\Omega}\langle\nabla\eta,V\rangle\big(\lambda\phi_i\phi_j-\langle\nabla\phi_i,\nabla\phi_j\rangle\big)\dm.
\end{eqnarray*}
This computation tells us that the last two terms in \eqref{eq-a_ij} cancel each other, which concludes the proof of the proposition.
\end{proof}

\section{Applications}

In this section, we concentrate on the applications of the Hadamard type formulas. We first prove the following:
\begin{proposition}\label{prop-A}
Let $(M,g_0)$ be a compact Riemannian manifold and $\lambda$ an eigenvalue of $L_{g_0}$ for the Dirichlet problem with multiplicity $m>1$. Take the positive number $\epsilon_{\lambda,g_0}$ and the neighborhood $\mathcal{V}_{\epsilon}$ of $g_0$ in $\mathcal{M}^r$ as in~\eqref{kato-continuity}. Then for each open neighborhood $\mathcal{U}\subset \mathcal{V}_{\epsilon}$ there is $g\in\mathcal{U}$ such that all eigenvalues $\lambda(g)$ of $L_g$ with $|\lambda(g)-\lambda|<\epsilon_{\lambda,g_0}$ are simple.
\end{proposition}
\begin{proof}
We argue by contradiction. Suppose that there is an open neighborhood $\mathcal{U}\subset \mathcal{V}_\epsilon$ of $g_0$ such that for all $g\in \mathcal{U}$ the eigenvalue $\lambda(g)$ of $L_{g}$ with $|\lambda(g)-\lambda|< \epsilon_{\lambda,g_0}$ has multiplicity $m>1$. In this case, for any symmetric $(0,2)$--tensor $T$ on $(M,g)$ the perturbation $g(t)=g+tT$ fails to split the eigenvalue $\lambda$. The eigenvalue curves $\lambda(t)$ satisfy
$$\left\{
\begin{array}{ccccccc}
-L_{g(t)}\phi_i(t)&=& \lambda(t)\phi_i(t) &\hbox{in} \,\,M\\
\phi_i(t)&=&0 &\hbox{on}\,\, \partial M.
\end{array}
\right.$$
Since $H=\frac{d}{dt}g(t)=T$ and $L=L_g$, by Proposition~\ref{pro bar-L} we have
\begin{equation}
\lambda'\delta_{ij} = \int_{M}\big\langle\frac{1}{4}L(\phi_i\phi_j)g - d\phi_i\otimes d\phi_j, T\big\rangle \dm.
\end{equation}
Now, considering the symmetrization tensor $S_{ij}=\frac{d\phi_i\otimes d\phi_j + d\phi_j\otimes d\phi_i}{2}$ and using the fact that
\begin{equation*}
\langle d\phi_i\otimes d\phi_j, T\rangle=\langle d\phi_j\otimes d\phi_i,T \rangle
\end{equation*}
we deduce the next identity
\begin{equation}\label{EQ_derivative}
\lambda'\delta_{ij} = \int_{M}\big\langle\frac{1}{4}L(\phi_i\phi_j)g - S_{ij}, T\big\rangle \dm.
\end{equation}
If $i\neq j$, we have
\begin{equation}\label{eq S}
\frac{1}{4}L(\phi_{i}\phi_{j})g= S_{ij}.
\end{equation}
Furthermore, taking the trace in equation \eqref{eq S}, we have
\begin{eqnarray}\label{eq-aux thm1}
\nonumber g(\nabla\phi_{i},\nabla\phi_{j}) &=& \frac{n}{4}L(\phi_{i}\phi_{j})=\frac{n}{4}(\phi_{i}L\phi_{j}+\phi_{j}L\phi_{i}+2g(\nabla\phi_{i},\nabla\phi_{j}))\\
&=&\frac{n}{2}(-\lambda\phi_i\phi_j + g(\nabla\phi_i,\nabla\phi_j)).
\end{eqnarray}
For $n\neq2$ we can write
\begin{equation}
\frac{n\lambda}{n-2}\phi_{i}\phi_{j}=g(\nabla\phi_{i},\nabla\phi_{j}).
\end{equation}
Fixing $p\in M$ we consider an integral curve $\alpha$ in $M$ such that $\alpha(0)=p$ and $\alpha'(s)=\nabla\phi_i(\alpha(s))$. Defining $\beta(s):=\phi_j(\alpha(s))$, we compute
\begin{eqnarray*}
\beta'(s)&=&\langle \nabla\phi_j(\alpha(s)),\alpha'(s)\rangle =g(\nabla\phi_j,\nabla\phi_i)(\alpha(s))=\frac{n\lambda}{n-2}\phi_{i}\phi_{j}(\alpha(s))\\
&=&\frac{n\lambda}{n-2}\phi_{i}(\alpha(s))\beta(s),
\end{eqnarray*}
which is a contradiction, since $M$ is compact. For the case $n=2$, we have from equation~\eqref{eq-aux thm1} that $\phi_i\phi_j=0$. Then, it follows from the principle of the unique continuation~\cite{Romander} that at least one of the eigenfunctions vanishes, which is again a contradiction. Therefore, we complete the proof of the proposition.
\end{proof}

\begin{proposition}\label{prop-B}
Let $(M,g)$ be a Riemannian manifold and $\Omega$ a bounded domain in $M$. Let $\lambda$ be an eigenvalue of $L_g$ for the Dirichlet problem with multiplicity $m>1$. Take the positive number $\epsilon_{\lambda,\Omega}$ and the neighborhood $\mathcal{V}_{\epsilon}$ of the identity in $D^r(\Omega)$ as in~\eqref{kato-continuity}. Then for any open neighborhood $\mathcal{U}\subset \mathcal{V}_{\epsilon}$, there is a diffeomorphism $f$ such that all eigenvalues $\lambda(f)$ with $|\lambda(f)-\lambda|< \epsilon_{\lambda,\Omega}$ are simple.
\end{proposition}
\begin{proof}
We also argue by contradiction. Suppose that there is an open neighborhood $\mathcal{U}\subset \mathcal{V}_\epsilon$ of the identity such that for all $f\in \mathcal{U}$ the eigenvalue $\lambda(f)$ of $L_{g}$ with $|\lambda(f)-\lambda|< \epsilon_{\lambda,\Omega}$ has multiplicity $m>1$. Then, it follows from~\eqref{eq-prop3} that $\frac{\partial\phi_i}{\partial \nu}\frac{\partial\phi_j}{\partial \nu} = 0$ on $\partial\Omega$. This way, we have either $\frac{\partial\phi_i}{\partial \nu} = 0$ or $\frac{\partial\phi_j}{\partial \nu} = 0 $ in some open set $U$ of $\partial \Omega$. If $\frac{\partial\phi_i}{\partial \nu} = 0$ in $U$, since $\phi_i = 0 $ on $\partial \Omega$, it follows from the unique continuation principle \cite{Romander} that $\phi_i = 0$ on $\Omega$, which is a contradiction.
\end{proof}

\subsection{Proof of Theorem~\ref{thm-A}}
\begin{proof}
Let $\mathcal{C}_m$ be the set of metrics in $\mathcal{M}^r$ such that the first $m$ eigenvalues of $L_g$ are simple. It is known that if these eigenvalues depend continuously on the metric (see~\cite{BU}), then for each $m$ the set $\mathcal{C}_m$ is open in $\mathcal{M}^r$. On the other hand, it follows from Proposition~\ref{prop-A} that the set  $\mathcal{C}_m$ is dense in $\mathcal{M}^r$. Since $\mathcal{M}^r$ is a complete metric space in the $C^r$~topology the set $\Gamma = \cap_{m=1}^{\infty}\mathcal{C}_m$ is dense, which proves this theorem.
\end{proof}

\subsection{Proof of Theorem~\ref{thm-B}}
\begin{proof}
Since $D^r(\Omega)$ is an affine manifold in a Banach space, similar arguments as above allow us to obtain Theorem~\ref{thm-B}.
\end{proof}

\subsection{Proof of Theorem~\ref{thm-f-A}}
\begin{proof}
The proof follows from the analogous steps for the variation of metrics case. We shall present a brief sketch of the last claim. Indeed, the main tool is to show a proposition analogous to Proposition~\ref{prop-A} for $\eta$-variation case. For this, first note that from equation~\eqref{Eq-ProBar-L} we get $(\lambda_i+\lambda_j)'\delta_{ij} = \int_{M}\langle\nabla\dot{\eta},\nabla(\phi_{i}\phi_{j})\rangle\dm$. Second, by using integration by parts formula, we obtain $2\lambda'\delta_{ij}= -\int_{M}\dot{\eta}L_\eta(\phi_i\phi_j)\dm$, since $\lambda_i=\lambda_j=\lambda$. Now, we argue as in the proof of the Proposition~\ref{prop-A} to get a contradiction, namely, the integral $\int_{M}\dot{\eta}L_\eta(\phi_i\phi_j)\dm=0$ for all $\dot{\eta}\in\mathcal{F}^r$. This is equivalent to $g(\nabla\phi_i,\nabla\phi_j) -\lambda\phi_i\phi_j=0$, but nontrivial eigenfunctions cannot satisfy it. Finally, we can proceed as in the proof of Theorem~\ref{thm-A}, and this completes our sketch.
\end{proof}

\subsection{Proof of Theorem~\ref{thm-f}}
\begin{proof}
Following the Teytel's approach as in Gomes and Marrocos~\cite[Section~$5$]{Gomes-Marrocos}, we define the linear functionals
\begin{equation*}
f_{ij}(\dot{\eta})=\int_{M}\phi_iL'_{\eta}\phi_j\dm,
\end{equation*}
where $\dot{\eta}\in \mathcal{F}^r$ and $L'_{\eta}=\frac{d}{dt}|_{t=0}L_{\eta(t)}=g(\nabla\dot{\eta},\nabla\cdot)$.

In order to prove Items $(1)$ and $(2)$ it is enough to verify that there exist two orthonormal eigenfunctions $\phi_1$ and $\phi_2$ associated to $\lambda$ such that the functionals $f_{11}-f_{22}$ and $f_{12}$ are linearly independent, see \cite[Remark~2]{Gomes-Marrocos} for details. However, we prove a more stronger condition, namely $f_{11}, f_{12}, f_{22}$ are linearly independent, so that we can apply the Implicit Function Theorem as in the proof of Theorem~$1.1$ in~\cite{teytel} to get Item~$(3)$ as well.

First of all, we use integration by parts formula to obtain
\begin{equation*}
f_{ij}(\dot\eta)= - \int_{M}\dot{\eta}L_\eta(\phi_i\phi_j)\dm.
\end{equation*}
So,
\begin{eqnarray*}
0=\alpha f_{11} +\beta f_{12} + \gamma f_{22} = \int_{M}\dot{\eta}(\alpha L_\eta(\phi_1^2)+\beta L_\eta(\phi_1\phi_2)+\gamma L_\eta(\phi_2^2))\dm.
\end{eqnarray*}
Whence, we conclude that
\begin{eqnarray*}
\alpha(|\nabla \phi_1|^2-\lambda\phi_1^2)+\beta g(\nabla\phi_1,\nabla\phi_2) -\lambda\phi_1\phi_2)+ \gamma(|\nabla \phi_2|^2-\lambda\phi_2^2)=0.
\end{eqnarray*}
Now, we can proceed as in~\cite[Subsection~$5.1$]{Gomes-Marrocos} to complete the proof of the theorem.
\end{proof}

\vskip .1cm
\noindent\textbf{Acknowledgements:}
The authors would like to express their sincere thanks to Dragomir M. Tsonev for useful comments, discussions and constant encouragement as well as to the Department of Mathematics at Lehigh University where part of this work was carried out. The first and second authors are grateful to Huai-Dong Cao and Mary Ann for their warm hospitality and constant encouragement. We should also like to acknowledge all the anonymous referees who tacitly contributed to the improvement of this work.

\end{document}